\documentclass[a4paper,12pt]{amsart}

\numberwithin{equation}{section}
\usepackage{amsmath}
\usepackage{amssymb}
\usepackage{amsthm}
\usepackage{comment}
\usepackage{color}
\usepackage{fancyhdr}

\newtheorem{thm}{Theorem}[section]
\newtheorem{prop}[thm]{Proposition}
\newtheorem{cor}[thm]{Corollary}
\newtheorem{lem}[thm]{Lemma}

\theoremstyle{definition}

 \newtheorem{nt}[thm]{Notation}
\newtheorem{rmk}[thm]{Remark}

\newcommand{\sgn}{{\rm sign}}
\newcommand{\cS}{\mathcal{S}}
\newcommand{\UMD}{{\rm UMD}}
\newcommand{\cL}{L}

\newcommand{\cH}{\mathcal{H}}
\newcommand{\cA}{\mathcal{A}}
\newcommand{\cB}{\mathcal{B}}
\newcommand{\cC}{\mathcal{C}}
\newcommand{\cI}{\mathcal{I}}

\title{The best constants for operator Lipschitz functions on Schatten classes}
\author{M. Caspers, S. Montgomery-Smith, D. Potapov, F. Sukochev}

\address{M. Caspers,  Laboratoire de Math\'ematiques, Universit\'e de Franche-Comt\'e, 16 Route de Gray, 25030 Besan\c con, France}
\email{martijn.caspers@univ-fcomte.fr}
\address{S. Montgomery-Smith, Mathematics Department Columbia, University of Missouri, MO 65211, USA}
\email{stephen@missouri.edu}
\address{D. Potapov, F. Sukochev, School of Mathematics and Statistics, UNSW, Kensington 2052, NSW, Australia}
\email{d.potapov@unsw.edu.au}
\email{f.sukochev@unsw.edu.au}

\thanks{\noindent \today.   \\
The first author was supported by the ANR project: ANR-2011-BS01-008-01.\\
The last two authors are partially supported by the ARC}

\begin{document}

\thispagestyle{empty}

\maketitle

\normalsize

\begin{abstract}
Suppose that $f$ is a Lipschitz function on $\mathbb{R}$ with $\Vert f \Vert_{{\rm Lip}}\leq 1$. Let $A$ be a bounded self-adjoint operator on a Hilbert space $\cH$.   Let $p \in (1, \infty)$ and suppose that $x \in B(\cH)$ is an operator such that the commutator $[A, x]$ is contained in the Schatten class $\cS_p$. It is proved by the last two authors, that then also $[f(A), x] \in \cS_p$ and there exists a constant $C_p$ independent of $x$ and $f$ such that
\[
\Vert [f(A), x] \Vert_p \leq C_p \Vert [A,x] \Vert_p.
\]
The main result of this paper is to give a sharp estimate for $C_p$ in terms of $p$. Namely, we show that $C_p \sim \frac{p^2}{p-1}$. In particular, this gives the best estimates  for operator Lipschitz inequalities.

We treat this result in a more general setting. This involves commutators of $n$ self-adjoint   operators $A_1, \ldots, A_n$, for which we prove the analogous result. The case described here in the abstract follows as a special case.
\end{abstract}

\section{Introduction}

Recently, the  last two authors proved that Lipschitz functions on $\mathbb{R}$ act as operator Lipschitz functions on the Schatten  classes $\cS_p$ for all $p \in (1, \infty)$, see \cite{PotSuk12}, \cite{Sal09}.   That is, suppose that $f: \mathbb{R} \rightarrow \mathbb{R}$ is a Lipschitz function and
\[
\Vert f \Vert_{{\rm Lip}} = \sup_{\xi, \tilde{\xi}\in \mathbb{R}}  \frac{\vert f(\xi) - f(\tilde{\xi}) \vert }{\Vert \xi - \tilde{\xi}\Vert_1}  \leq 1.
\]
Let $p \in (1, \infty)$. Suppose that $A, B$ are bounded, self-adjoint operators such that $A-B \in \cS_p$. Then, it was proved in \cite{PotSuk12} that also $f(A) - f(B) \in \cS_p$ and there is a constant $C_p< \infty$ independent of $A, B$ and $f$ such that
\begin{equation}\label{EqnIntroLipschitz}
\Vert f(A) - f(B) \Vert_p \leq C_p \Vert A - B \Vert_p.
\end{equation}
We denote $C_p$ for the minimal constant for which the inequality (\ref{EqnIntroLipschitz}) holds.

For the case $p = 1$, the analogous result fails. That is, there is no constant $C_1$ such that the inequality (\ref{EqnIntroLipschitz}) holds as was proved in \cite{Dav88}. For the case $p = \infty$ the analogous statement also fails as was proved in \cite{Kat73}.

This raises the question of what the growth order of $C_p$ is as $p$ approaches either $1$ or $\infty$. In \cite{KPSS} it was proved that $C_p \preceq p^8$ as $p \rightarrow \infty$ and $C_p \preceq (p-1)^{-8}$ as $p \downarrow 1$. In fact, in \cite{KPSS} a more general result is covered involving an $n$-tuple of commuting self-adjoint (bounded) operators. We refer to \cite[Theorem 5.3]{KPSS} for the precise statement.

In \cite{PotSuk12} an estimate for the asymptotic behavior of $C_p$ was not mentioned explicitly. However, it is in principle possible to find an upper estimate for $C_p$ from the proof presented in \cite{PotSuk12}. These proofs involve the Marcinkiewicz multiplier theorem as well as diagonal truncation and do not lead to a sharp upper estimate of $C_p$.

The main result of this paper is a sharp estimate for $C_p$.  Namely,   we prove that $C_p \sim p$ as $p \rightarrow \infty$ and we prove that $C_p \sim (p-1)^{-1}$ as $p \downarrow 1$.  Our result is stated in terms of commuator estimates in Schatten classes. In particular, it sharpens the estimates found in \cite{KPSS} for $n$-tuples of commuting self-adjoint operators. Note that the lower estimate of this result can be found in \cite{Abd90}.

The novelty of our proof is that we apply the main result of \cite{GeiMonSak}. In \cite{GeiMonSak} sharp estimates were found for the action of a smooth, even multiplier that acts on vector valued $L^p$-spaces. The norm of such a multiplier can be expressed in terms of the UMD-constant of a Banach space (we recall the definition below). This result together with the so-called transference method forms the key argument that allows us to improve the known estimates for $C_p$.

This paper relates to the general interest of finding the best constants in non-commutative probability inequalities. In particular, major achievements have been made considering the best constants of Burkholder/Gundy inequalities \cite{JunXu05}, \cite{Ran05},  Doob and Stein inequalities \cite{JunXu05} and Khintchine inequalities \cite{Haa81}, \cite{HaaMus07}.

\vspace{0.3cm}

The structure of this paper is as follows. Section \ref{SectMultipliers} recalls the necessary theory on Fourier multipliers. In Section \ref{SectSpecialMultiplier} we construct a special multiplier that forms a key step for our main result. In Section \ref{SectDoubleOpInt} we recall the theory of double operator integrals and prove the necessary lemmas on discrete approximations.  Section \ref{SectCommutatorEstimate} contains our main result and the core of our proof.

\subsection*{General conventions}

For $p \in [1, \infty)$ we write $\cS_p$ for the Schatten-von Neumann classes. These are the non-commutative $L^p$-spaces associated with the bounded operators on a Hilbert space $\cH$ with respect to the standard trace $\tau$. For $p \in (1, \infty)$ we use $p' \in (1,\infty)$ to denote the conjugate exponent, which is defined by $\frac{1}{p}+\frac{1}{p'}=1$.
We use $\chi$ to denote an indicator function.

Let $C_p, D_p \in \mathbb{R}^+$ be constants depending on $p \in (1, \infty)$. We write
\[
C_p \sim D_p,
\]
if there are constants $a, b$ such that $a \leq C_p / D_p \leq b$ for all $p \in (1, \infty)$. In particular, in this case $C_p$ and $D_p$ have the same asymptotic behavior as $p \rightarrow \infty$ or $p \downarrow 1$.

\section{Construction of Fourier multipliers}\label{SectMultipliers}

We recall the theory of $L^p$-Fourier multipliers and their vector valued counterparts.
In this section $\xi = (\xi_1, \ldots, \xi_n)$  is always a vector in $\mathbb{R}^n$. $\mu$ is always   a number in $\mathbb{R}$.

\subsection{Multipliers} A function $m \in \cL^\infty(\mathbb{R}^n)$ defines a bounded linear map:
\[
T_m: \cL^2(\mathbb{R}^n) \rightarrow \cL^2(\mathbb{R}^n): f \mapsto ( \mathcal{F}_2^{-1} \circ m \circ \mathcal{F}_2  )(f).
\]
Here, $\mathcal{F}_2: \cL^2(\mathbb{R}^n) \rightarrow \cL^2(\mathbb{R}^n )$ is the Fourier transform that is defined for $f \in \cL^1(\mathbb{R}^n) \cap \cL^2(\mathbb{R}^n)$ as:
\[
(\mathcal{F}_2 f)(\xi) = \frac{1}{\sqrt{2\pi}} \int_{\mathbb{R}^n} f(s) e^{-i s \cdot \xi} ds.
\]
Let $p \in [1, \infty)$. Suppose that for every $f \in \cL^2(\mathbb{R}^n) \cap \cL^p(\mathbb{R}^n)$ we have $T_m(f) \in \cL^p(\mathbb{R}^n)$ and
\begin{equation}\label{EqnLpMultiplier}
T_m: \cL^2(\mathbb{R}^n) \cap \cL^p(\mathbb{R}^n) \rightarrow \cL^p(\mathbb{R}^n)
\end{equation}
 extends to a bounded map on $\cL^p(\mathbb{R}^n)$. Then, we call $m$ an $\cL^p$-multiplier and we keep denoting the extension of (\ref{EqnLpMultiplier}) by $T_m$. $m$ is called {\it homogeneous} (of degree 0) if for every $\lambda \in \mathbb{R}^+$ we have $m(\lambda \xi) =  m(\xi)$. We call $m$ {\it even} if $m(-\xi) = m(\xi)$. $m$ is called {\it odd} if $m(-\xi) = - m(\xi)$. A homogeneous $\cL^p$-multiplier is called smooth if it is smooth on the unit sphere in $\mathbb{R}^n$ (or equivalently, if it is smooth on $\mathbb{R}^n \backslash \{ 0 \}$).

\begin{rmk}\label{RmkAlgebra}
Let $p \in [1, \infty)$. By definition, the set of $\cL^p$-multipliers  forms an algebra.
\end{rmk}

\subsection{Vector valued multipliers} We are mostly concerned with vector valued counterparts of $\cL^p$-multipliers.
Let $E$ be a Banach space. Let $(X, \mu)$ be a $\sigma$-finite measure space and $p \in [1, \infty)$. Let $\cL^p_E(X) = \cL^p_E(X,\mu) $ denote the space of strongly measurable functions $f:  X \rightarrow E$ for which there is a separable subspace $E_0 \subseteq E$ such that $f(X) \subseteq E_0$ and
\[
\Vert f \Vert_{\cL^p_E(X)} := \left(\int_X \Vert f(x) \Vert^p d\mu(x)\right)^{\frac{1}{p}} < \infty.
\]
 Let $T: \cL^p(X)  \rightarrow \cL^p(X) $ be a bounded operator. Then, $T$ defines a linear map on the simple functions of $\cL^p_E(X)$ which we denote for the moment by
\[
 T(E): \sum_{k =1 }^n x_k \chi_{A_k} \mapsto  \sum_{k =1 }^n x_k T( \chi_{A_k} ), \qquad x_k \in E,\: A_k \subseteq X \textrm{ measurable}.
\]
In case,
\[
  \sup \left\{ \Vert T(E) \left( f \right) \Vert \mid f \textrm{ simple and } \Vert f \Vert_{\cL^p_E(X)} \leq 1
\right\}
\]
is finite, the map $T(E)$ extends to a bounded map $T(E): \cL^p_E(X) \rightarrow \cL^p_E(X)$.

\begin{nt}
We write $T$ for $T(E): \cL^p_E(X) \rightarrow \cL^p_E(X)$. It will always be clear from the context if $T$ acts on $\cL^p(X)$ or $\cL^p_E(X)$.
\end{nt}

In particular, we apply the previous   construction to the special case where $(X, \mu)$ is $\mathbb{R}^n$ equipped with the Lebesgue measure, $E$ is a non-commutative $\cL^p$-space and $T$ is $T_m$ for some $\cL^p$-multiplier $m$.


\begin{rmk}\label{RmkCoordintateTransform}
Suppose that $m$ is a $\cL^p$-multiplier such that $T_m: \cL^p_E(\mathbb{R}^n) \rightarrow \cL^p_E(\mathbb{R}^n)$ is bounded. Let $A: \mathbb{R}^n \rightarrow \mathbb{R}^n$ be a linear invertible transformation. Then, also $T_{m \circ A}:  \cL^p_E(\mathbb{R}^n) \rightarrow \cL^p_E(\mathbb{R}^n)$ is bounded with the same norm. This follows from the observation (see also \cite[p. 557]{GeiMonSak}):
\[
(T_{m \circ A}f)(\xi) = (T_m( f\circ A^{T}) )((A^{T})^{-1} \xi ).
\]
Here, $A^T$ is the transpose  of $A$.
\end{rmk}
\begin{rmk}\label{RmkIntegral}
Let $GL_n(\mathbb{R})$ denote the invertible $n\times n$-matrices. Let $A\in GL_n(\mathbb{R})$. For $f \in \cL^p_E(\mathbb{R}^n)$, let $A^\ast(f) = f \circ A$. So $A^\ast$ defines a bounded map on $\cL^p_E(\mathbb{R}^n)$,  (with bound given by the determinant of $A^{-1}$, as follows from a substitution of variables). Let
\[
[0, 1] \rightarrow GL_n(\mathbb{R}): t \mapsto A_t
\]
 be a continuous path. Then, $t \mapsto A_t^\ast$ is a strongly continous path with values in the bounded operators on $\cL^p_E(\mathbb{R}^n)$. Indeed, one can check that for simple functions $f \in \cL^p_E(\mathbb{R}^n)$ the path $t \mapsto A_t^\ast(f)$ is continous and then use the fact that $t \mapsto \Vert A_t^\ast\Vert$ is bounded. Let $m$ be a $\cL^p$-multiplier and let $E$ be a Banach space such that $T_m: \cL^p_E(\mathbb{R}^n) \rightarrow  \cL^p_E(\mathbb{R}^n)$  is bounded. Then, the strong integral $\int_0^1 T_{m \circ A_{t}} dt$ exists and defines a bounded map on $\cL^p_E(\mathbb{R}^n)$.
\end{rmk}

\subsection{Discrete multipliers}

 Let $m \in \cL^\infty(\mathbb{R}^n)$ be an  odd  $\cL^p$-multiplier. Put $\bar{m} \in \cL^\infty(\mathbb{Z}^n)$ by setting   $\bar{m}(k) = m(k), k \in \mathbb{Z}^n \backslash \{0\}$ and $\bar{m}(0) = 0$. For $f$ a finite trigonometric polynomial, we set
\[
(T_{\bar{m}}f)(\theta) =  \sum_{k \in \mathbb{Z}^n} \bar{m}(k) \hat{f}(k) e^{ik\theta}.
\]
Here, $\hat{f}(k) = \int_{\mathbb{T}^n} f(\theta) e^{-i k \cdot \theta} d\theta$, where the $n$-torus $\mathbb{T}^n$ is considered with the normalised Lebesgue measure. The following theorem gives a   sufficient condition on $m$ in order to extend $T_{\bar{m}}$ to a bounded map on $\cL^p(\mathbb{T}^n)$.

\begin{thm}[Theorem 3.6.7 of \cite{GrafakosI}]\label{ThmDiscrete}
Let $m \in \cL^\infty(\mathbb{R}^n)$ be an   odd  function that is smooth on $\mathbb{R}^n \backslash \{ 0 \}$. Let $p \in [1, \infty)$ and suppose that $m$ is a $\cL^p$-multiplier.
Let $E$ be a Banach space and suppose that $T_m: \cL^p_E(\mathbb{R}^n) \rightarrow \cL^p_E(\mathbb{R}^n)$ is bounded. Then,
\begin{equation}\label{EqnDiscreteVsContinuous}
\Vert T_{\bar{m}}: \cL^p_E(\mathbb{T}^n) \rightarrow \cL^p_E(\mathbb{T}^n) \Vert \leq
\Vert T_{m}: \cL^p_E(\mathbb{R}^n) \rightarrow \cL^p_E(\mathbb{R}^n) \Vert.
\end{equation}
\end{thm}

\begin{rmk}
Theorem \ref{ThmDiscrete} was proved for $E = \mathbb{C}$ in \cite[Theorem 3.6.7]{GrafakosI}. For a general Banach space $E$, the statement follows from a mutatis mutandis copy of its proof.
The exact statement of Theorem \ref{ThmDiscrete} can also be found as \cite[Lemma 2.2]{GeiMonSak}.
\end{rmk}

\subsection{The UMD-property}

Let $E$ be a Banach space. $E$ is said to have the {\it UMD-property} (Unconditional Martingale Differences) if there exists a constant $C_p(E)$ with $p \in (1, \infty)$ such that for every probability measure space $(\Omega, \Sigma, \mu)$ and every sequence of $\sigma$-subalgebras $B_1 \subseteq B_2 \subseteq \ldots \subseteq \Sigma$ and every martingale difference sequence $\{ d_n \}_{n=1}^\infty$ with respect to $\{ B_n\}_{n=1}^\infty$ in $\cL^p_E(\Omega)$, the sequence  $\{ d_n \}_{n=1}^\infty$ satisfies:
\begin{equation}\label{EqnUMD}
\Vert \sum_{k=1}^n \epsilon_k \alpha_k d_k \Vert_{\cL^p_E(\Omega)} \leq C_p(E) \Vert \sum_{k=1}^n   \alpha_k d_k \Vert_{\cL^p_E(\Omega)}
\end{equation}
for every $\epsilon_k = \pm 1$ and scalars $\{ \alpha_k \}_{k=1}^\infty$ and all $n = 1, 2, \ldots$. We will denote the minimal constant $C_p(E)$ for which (\ref{EqnUMD}) holds by $\UMD_p(E)$. This constant is also called the {\it UMD-constant} of $E$.

\begin{thm}[Theorem 4.3 and Remark 4.4 of \cite{Ran02}]\label{ThmUMDConstantSp}
The Schatten class $\cS_p$ is a UMD-space for every $p \in (1, \infty)$. Moreover,
\[
\UMD_p(\cS_p) \sim \frac{p^2}{p-1}.
\]
\end{thm}

\begin{rmk}
Theorem \ref{ThmUMDConstantSp} is also valid if $\cS_p$ is replaced by a non-commutative $\cL^p$-space assocated with an arbitrary von Neumann algebra $M$. This particularly applies to Haagerup $\cL^p$-spaces associated with a non-semi-finite von Neumann algebra  $M$, see \cite{Ran02}. For Haagerup $L^p$-spaces we refer to \cite{Haa77}, \cite{Ter81}.
\end{rmk}

\subsection{The Hilbert transform}\label{SectHilbertTransform}
Consider the function $h: \mathbb{R}^n \rightarrow \mathbb{C}: \xi \mapsto i \: \sgn(\xi_1)$ where we use the convention $\sgn(0) = 0$. Let $E$ be a UMD-space. Then, for every $p \in (1, \infty)$,
\begin{equation}\label{EqnHilbertTransform}
T_h: \cL^p_E(\mathbb{R}^n) \rightarrow \cL^p_E(\mathbb{R}^n)
\end{equation}
is bounded. In fact, $E$ is a UMD-space if and only if (\ref{EqnHilbertTransform}) is bounded for every $p \in (1, \infty)$ \cite{Bou83}, \cite{Bur83}. $T_h$ is also called the {\it Hilbert transform}, see also  \cite[Chapter 4]{GrafakosI}.


\subsection{The Riesz transform} \label{SectRiesz}
Consider the function $r_j: \mathbb{R}^n \rightarrow \mathbb{C}: \xi \mapsto i \frac{\xi_j}{\Vert \xi \Vert_2}$ where we use the assumption $r_j(0) = 0$. Let $E$ be a UMD-space. Then, for every $p \in (1, \infty)$,
\begin{equation}\label{EqnRieszTransform}
T_{r_j}: \cL^p_E(\mathbb{R}^n) \rightarrow \cL^p_E(\mathbb{R}^n)
\end{equation}
is bounded.

$T_{r_j}$ is also called the {\it Riesz transform}, see also \cite[Chapter 4]{GrafakosI}.

\section{Construction of a special multiplier}\label{SectSpecialMultiplier}
The goal of this section is to construct a specific   smooth homogeneous even multiplier $m_{j}$. This multiplier plays an essential role in Section \ref{SectCommutatorEstimate}.   In this section $\xi = (\xi_1, \ldots, \xi_n)$  is always a vector in $\mathbb{R}^n$. $\mu$ is always   a number in $\mathbb{R}$.

\begin{lem}\label{LemFirstMultiplier}
For $1 \leq j \leq n$, there exists a function $m_{1,j}$ on $\mathbb{R}^{n+1}$ such that
\begin{equation}\label{EqnFirstMultiplier}
m_{1,j}(\xi, \mu) = \left\{
\begin{array}{ll}
  \frac{\mu}{\Vert \xi \Vert_2} \frac{\xi_j}{\Vert \xi\Vert_2}& \textrm{ if } \frac{\vert \mu \vert}{ \Vert \xi \Vert_1} < 1, \\
0 &   \textrm{ if } \frac{\vert \mu \vert}{ \Vert \xi \Vert_1} > 1,
\end{array}
\right.
\end{equation}
and moreover, such that for every $p \in (1, \infty)$ and UMD-space $E$, the map $T_{m_{1,j}}: \cL^p_E(\mathbb{R}^{n+1}) \rightarrow \cL^p_E(\mathbb{R}^{n+1})$ is bounded.
\end{lem}
\begin{proof}
Recall that we use the convention $\sgn(0) = 0$. Let $h(\xi) = i \: \sgn(\xi_1)$, see also Section \ref{SectHilbertTransform}. We start with observing that
\begin{equation}\label{EqnHilbertIndicator}
 \frac{1}{2} ( -i h(\xi)+1) =
\left\{
\begin{array}{ll}
1 & \textrm{ if } \xi_1 > 0, \\
\frac{1}{2} & \textrm{ if } \xi_1 = 0, \\
0 & \textrm{ if } \xi_1 < 0.
\end{array}
\right.
\end{equation}
Hence, $ \frac{1}{2} ( -i h(\xi)+1)$ is equal to the indicator function $\chi_{[0,\infty)}(\xi_j)$ except for the point 0.  Let
\[
h_{a,b}(\xi, \mu) = i \: \sgn(b \xi_1 + \ldots + b\xi_n + a \mu)
\]
 be the multiplier associated with the Hilbert transform  (\ref{EqnHilbertTransform}) precomposed with the linear map
\[
A_{a,b}: (\xi_1, \ldots, \xi_n, \mu) \mapsto (b \xi_1 + \ldots + b\xi_n + a \mu, \xi_2, \ldots, \xi_n, \mu),
\]
 see Remark \ref{RmkCoordintateTransform}. Put $k = \frac{1}{2} \int_{-1}^1 h_{1,t} dt$. So,
\begin{equation}\label{EqnKexpression}
k(\xi, \mu) = \frac{1}{2} i \int_{-1}^1 \sgn(t\xi_1+ \ldots + t\xi_n + \mu) dt.
\end{equation}
If $\vert \mu \vert \geq \vert \xi_1 + \ldots + \xi_n \vert$, then $\sgn(t\xi_1+ \ldots + t\xi_n +\mu) = \sgn(\mu)$ for every $t \in [-1,1]$. If $\vert \mu \vert \leq \vert \xi_1 + \ldots + \xi_n\vert$, then (\ref{EqnKexpression}) is equal to
\[
\begin{split}
k(\xi, \mu) = &\frac{1}{2} i  \int_{-1}^{-\frac{\mu}{\xi_1 + \ldots +\xi_n}} \!\!\!\!\!  \sgn(t\xi_1 + \ldots + t\xi_n + \mu)dt +  \frac{1}{2} i  \int^{1}_{-\frac{\mu}{\xi_1 + \ldots +\xi_n}} \!\!\!\!\!  \sgn(t\xi_1 + \ldots + t\xi_n + \mu)dt   \\
= & \frac{1}{2} i \int_{-1}^{-\frac{\mu}{\xi_1 + \ldots +\xi_n}} \!\!\!\!\! \sgn(-\xi_1 - \ldots - \xi_n) dt
  +  \frac{1}{2} i \int^{1}_{-\frac{\mu}{\xi_1 + \ldots +\xi_n}} \!\!\!\!\! \sgn( \xi_1 + \ldots + \xi_n) dt \\
= & i \frac{\mu}{\vert \xi_1 + \ldots + \xi_n\vert}.
\end{split}
\]
So  we conclude,
\[
k(\xi, \mu) = \left\{
\begin{array}{ll}
i \frac{\mu}{\vert \xi_1 + \ldots + \xi_n \vert} & \textrm{ if } \vert \mu \vert \leq \vert \xi_1 + \ldots + \xi_n \vert,\\
i \: \sgn(\mu) & \textrm{ if } \vert \mu \vert \geq \vert \xi_1 + \ldots + \xi_n \vert.
\end{array}
\right.
\]
 For $\epsilon = (\epsilon_1, \ldots, \epsilon_n) \in \{ -1, 1 \}^n$, put
\[
k_{\epsilon}(\xi, \mu) = k(\epsilon_1 \xi_1, \ldots, \epsilon_1 \xi_1, \mu) \cdot \Pi_{j=1}^n  \frac{1}{2} (  -i h(\epsilon_j \xi_j)+1).
\]
  We can explicitly describe $k_\epsilon$. The next formula can be determined by first considering the value of $k_{\epsilon}(\xi, \mu)$ for  the case that for every $1 \leq j \leq n$ we have $\xi_j \geq 0$. In that case, keeping in mind (\ref{EqnHilbertIndicator}), one arrives at equation (\ref{EqnExplicitK}) below.  Similarly, we can compute $k_\epsilon(\xi, \mu)$ for other signs of $\xi_j$.  This results in the following expression:
\begin{equation}\label{EqnExplicitK}
k_{\epsilon}(\xi, \mu) = \left\{
\begin{array}{ll}
i  \frac{1}{2^{l(\xi)}} \frac{\mu}{\vert \xi_1 \vert + \ldots + \vert \xi_n \vert}& \textrm{ if } \vert \mu \vert \leq \vert \xi_1 \vert + \ldots + \vert \xi_n \vert  \textrm{ and } \forall j: \: \epsilon_j \xi_j \geq 0,\\
i \: \frac{1}{2^{l(\xi)}}\: \sgn(\mu)   & \textrm{ if } \vert \mu \vert \geq \vert \xi_1 \vert + \ldots + \vert \xi_n \vert \textrm{ and } \forall j: \: \epsilon_j \xi_j \geq 0, \\
0 & {\rm else},
\end{array}
\right.
\end{equation}
where $l(\xi)$ is the number of coordinates $j$ for which $\xi_j = 0$.
Put $K = \sum_{\epsilon \in \{ -1, 1\}^n} k_{\epsilon}$. Then, treating again the different possibilities for the signs of $\xi_j$ separately, one computes
\begin{equation}\label{EqnExplicitBigK}
K(\xi, \mu) = \left\{
\begin{array}{ll}
i \frac{ \mu}{\Vert \xi \Vert_1} & \textrm{ if } \vert \mu \vert \leq \Vert \xi \Vert_1, \\
i\: \sgn(\mu) & \textrm{ if } \vert \mu \vert > \Vert \xi \Vert_1.
\end{array}
\right.
\end{equation}

Next, for $\epsilon = (\epsilon_1, \ldots, \epsilon_n) \in \{ -1, 1\}^n$. Consider the multiplier:
\[
\begin{split}
r_\epsilon(\xi, \mu) = & i \frac{\epsilon \cdot \xi}{\Vert \xi \Vert_2} \cdot \Pi_{j=1}^n \frac{1}{2} (  -i h(\epsilon_j \xi_j)+1) \: \cdot\: \\
 & \left(  \frac{1}{2} (-i  h(\mu)+1)  \:  \frac{1}{2} (-i  h (\epsilon \cdot \xi - \mu)+1) + \frac{1}{2} (  -i h(-\mu)+1)  \: \frac{1}{2} (  -i h(\epsilon \cdot \xi + \mu) +1) \right).
\end{split}
\]
Note that for $(\xi, \mu)$ to be in the support of $r_\epsilon$, we must have for all $1 \leq j \leq n$ that $\sgn(\xi_j) = \epsilon_j$ or $\xi_j = 0$. Moreover, every $(\xi, \mu)$ in the support of $r_\epsilon$ must satisfy $\Vert \xi \Vert_1 \geq \mu$ in case $\mu \geq 0$ and $\Vert \xi \Vert_1 \leq \mu$ in case $\mu \leq 0$.
Then,  with $l(\xi)$ as before, and taking into account that $\epsilon \cdot \xi = \Vert \xi \Vert_1$,
\[
 r_\epsilon(\xi, \mu) =
\left\{
\begin{array}{ll}
i  \frac{1}{2^{l(\xi)}}  \frac{\Vert \xi \Vert_1}{\Vert \xi \Vert_2} & \textrm{ if } \vert \mu \vert < \Vert \xi \Vert_1 \textrm{ and } \forall j: \epsilon_j \xi_j \geq 0,\\
0 &  \textrm{ if } \vert \mu \vert > \Vert \xi \Vert_1 \textrm{ or } \exists j: \epsilon_j \xi_j < 0.
\end{array}
\right.
\]
Put $R = \sum_{\epsilon \in \{ -1, 1\}^n} r_{\epsilon}$. Then,
\begin{equation}\label{EqnExplicitBigR}
R(\xi, \mu) =
\left\{
\begin{array}{ll}
i \frac{\Vert \xi \Vert_1}{\Vert \xi \Vert_2} &  \textrm{ if } \vert \mu \vert < \Vert \xi \Vert_1,\\
0 &  \textrm{ if } \vert \mu \vert > \Vert \xi \Vert_1.
\end{array}
\right.
\end{equation}
Recall the Riesz transform $r_j$ from Section \ref{SectRiesz}. We define the multiplier,
\[
m_{1,j}(\xi, \mu) = i K(\xi, \mu) R(\xi, \mu) r_j(\xi).
\]
 Then, it follows from (\ref{EqnExplicitBigK}), (\ref{EqnExplicitBigR}) and (\ref{EqnRieszTransform}) that $m_{1,j}$ satisfies (\ref{EqnFirstMultiplier}).



Let $E$ be a UMD-space and let $p \in (1, \infty)$. By construction of $m_{1,j}$ the map
\[
T_{m_{1,j}}: \cL^p_E(\mathbb{R}^{n+1}) \rightarrow \cL^p_E(\mathbb{R}^{n+1})
\]
 is bounded as follows from Remarks \ref{RmkAlgebra} and \ref{RmkIntegral}   and the fact that the Hilbert transform and Riesz transform are bounded operators on $\cL^p_E(\mathbb{R}^n)$.

\end{proof}

\begin{lem}\label{LemMultiplierEven}
There exists a  homogeneous even function $m_{j}$ on $\mathbb{R}^{n+1}$ that is smooth on $\mathbb{R}^{n+1} \backslash \{ 0\}$ such that
\begin{equation}\label{EqnEvenMultiplier}
m_{j}(\xi, \mu) =
\begin{array}{ll}
  \frac{\mu}{\Vert \xi \Vert_2} \frac{\xi_j}{\Vert \xi \Vert_2} & \textrm{ if } \frac{\vert \mu \vert}{ \Vert \xi \Vert_1} \leq 1,
\end{array}
\end{equation}
and moreover, for every $p \in (1, \infty)$ and UMD-space $E$, the map $T_{m_{j}}: \cL^p_E(\mathbb{R}^{n+1}) \rightarrow \cL^p_E(\mathbb{R}^{n+1})$ is bounded.
\end{lem}
\begin{proof}

Let $m_{1,j}$ be the $\cL^p$-multiplier of Lemma \ref{LemFirstMultiplier}. For every $\lambda \in \mathbb{R}^+$ the function $m_{\lambda,j} := \frac{1}{\lambda}m_{1,j}(\xi, \lambda \mu)$ is also a $\cL^p$-multiplier, see Remark \ref{RmkCoordintateTransform}. Note that,
\begin{equation}\label{EqnTranslatedMultiplier}
m_{\lambda,j}(\xi, \mu) = \left\{
\begin{array}{ll}
 \frac{\mu}{\Vert \xi \Vert_2}\frac{\xi_j}{\Vert \xi \Vert_2} & \textrm{ if } \lambda \: \vert \mu \vert  < \Vert \xi \Vert_1, \\
0 &   \textrm{ if } \lambda \: \vert \mu \vert > \Vert \xi \Vert_1.
\end{array}
\right.
\end{equation}

Let $s: [0, 1] \rightarrow [0,\infty)$ be a smooth function with support contained in $[\frac{1}{2}, \frac{3}{4}]$ and $\int_0^{1} s(\theta) d\theta = 1$. Set,
\begin{equation}\label{EqnSmoothing}
m_{j}   = \int _0^{1} s(\lambda) m_{\lambda,j}\: d\lambda.
\end{equation}
Consider the areas
\[
\begin{split}
A_1 = &  \left\{ (\xi, \mu) \in \mathbb{R}^{n+1} \mid  \Vert \xi \Vert_1   < \frac{1}{2}  \vert \mu \vert  \right\}, \\
A_2 = & \left\{ (\xi, \mu) \in \mathbb{R}^{n+1} \mid  \Vert \xi \Vert_1  > \frac{3}{4} \vert \mu \vert \right\}, \\
A_3 = & \left\{ (\xi, \mu)\in \mathbb{R}^{n+1}  \mid  \frac{1}{3}\vert\mu\vert <  \Vert \xi \Vert_1 < \vert\mu\vert   \right\}.
\end{split}
\]
We have $A_1 \cup A_2 \cup A_3 = \mathbb{R}^{n+1} \backslash \{ 0 \}$. Now, we check that $m_j$ is smooth on each of these areas.
For $(\xi, \mu) \in A_1$, we find that $m_{j}(\xi, \mu) = 0$ as follows from (\ref{EqnTranslatedMultiplier}) together with the fact that the support of $s$ is contained in $[\frac{1}{2}, \frac{3}{4}]$. So $m_j$ is smooth in $A_1$.
For $(\xi, \mu) \in A_2$, we find that
\[
m_{j}(\xi, \mu) =  \frac{  \mu }{\Vert \xi \Vert_2} \frac{ \xi_j }{\Vert \xi \Vert_2}.
\]
Indeed, this follows again from  (\ref{EqnTranslatedMultiplier}) together with the fact that the support of $s$ is contained in $[\frac{1}{2}, \frac{3}{4}]$.
 So $m_j$ is smooth on $A_2$. Since every $(\xi, \mu) \in A_2$ satisfies $\frac{\vert \mu \vert}{\Vert \xi \Vert_1} \leq 1$, this also proves that $m_{j}$ satisfies (\ref{EqnEvenMultiplier}). Define $S(t) = \int_0^t s(\lambda) d\lambda$, which is a smooth function on the open interval $(0,1)$. For $(\xi, \mu) \in A_3$ we find that
\[
\begin{split}
m_{j}(\xi, \mu) = & \int_{0}^{\frac{\Vert \xi \Vert_1}{\vert \mu \vert}} s(\lambda) m_{\lambda,j}(\xi, \mu) d\lambda +   \int_{\frac{\Vert \xi \Vert_1}{\vert \mu \vert}}^{1} s(\lambda) m_{\lambda,j}(\xi, \mu) d\lambda   \\
= & \int_{0}^{\frac{\Vert \xi \Vert_1}{\vert \mu \vert}} s(\lambda) d\lambda \: \cdot \: \frac{\mu}{\Vert \xi \Vert_2}  \frac{\xi_j}{\Vert \xi \Vert_2} + 0 \\
= &  S\left( \frac{\Vert \xi \Vert_1}{\vert \mu \vert}\right) \frac{\mu}{\Vert \xi \Vert_2} \frac{\xi_j}{\Vert \xi \Vert_2}.
\end{split}
\]
Here, the second equality follows from (\ref{EqnTranslatedMultiplier}). The other equalities follow from the definitions.
Hence, we see that  $m_{j}$ is smooth on $A_3$. We conclude that $m_j$ is smooth on $\mathbb{R}^{n+1}\backslash \{0\}$.

Let $E$ be a UMD-space and let $p \in (1, \infty)$. In Lemma \ref{LemFirstMultiplier} we proved that $T_{m_{1,j}}: \cL^p_E(\mathbb{R}^n) \rightarrow \cL^p_E(\mathbb{R}^n)$ is bounded. The definition of $T_{m_{j}}$  together with Remark \ref{RmkIntegral} implies that also $T_{m_{j}}: \cL^p_E(\mathbb{R}^n) \rightarrow \cL^p_E(\mathbb{R}^n)$ is bounded. Since $m_{1,j}$ is  even, also $m_j$ is even.
\end{proof}

The following theorem forms the key step in finding the best constants for commutator estimates in Schatten classes.

\begin{thm}[see Theorem 3.1 of \cite{GeiMonSak}]\label{ThmUMDEstimate}
Let $p \in (1, \infty)$ and let $E$ be a UMD-space. There exists a constant $C$ that is independent of $p$, $j$ and $E$ such that:
\[
    \Vert T_{m_{j}}: \cL^p_{E}(\mathbb{R}^{n+1}) \rightarrow \cL^p_E(\mathbb{R}^{n+1})\Vert \leq C \cdot \UMD_p(E).
\]
\end{thm}
\begin{rmk}
Let $E$ be a UMD-space and let $p \in (1, \infty)$. Using \cite[Proposition 3.8]{GeiMonSak} it follows directly that for any homogenenous even function $m \in \cL^\infty(\mathbb{R}^{n+1})$ that is smooth on $\mathbb{R}^{n+1}\backslash \{ 0\}$, the transform $T_{m}: \cL^p_E(\mathbb{R}^{n+1}) \rightarrow \cL^p_E(\mathbb{R}^{n+1})$ is bounded. This observation could be used to supply an alternative proof of Lemma \ref{LemMultiplierEven}. Here, we have chosen to give a self-contained proof.
\end{rmk}

\section{Double operator integrals}\label{SectDoubleOpInt}

The goal of this section is to recall the basic notions of double operator integrals \cite{PWS02}. We prove the necessary results in order to see that certain double operator integrals may be approximated by discrete versions of double operator integrals.
Troughout this section, $\xi = (\xi_1, \ldots, \xi_n), \tilde{\xi} = (\tilde{\xi}_1, \ldots, \tilde{\xi}_n)$ are vectors in $\mathbb{R}^n$.   Recall that $\tau$ denotes the standard semi-finite trace on the bounded operators on a Hilbert space $\cH$.

\vspace{0.3cm}

Let $E$  be  a spectral measure on $\mathbb{R}^n$ having compact support  taking values in the orthogonal projections on a Hilbert space $\cH$. $E$ generates an $n$-tuple of commuting self-adjoint bounded  operators
\begin{equation}\label{EqnArowDefinition}
\cA = (A_1, \ldots, A_n), \qquad \textrm{ with } A_k = \int_{\mathbb{R}^n} \xi_k dE(\xi).
\end{equation}
We also set
\[
f(\cA) = \int_{\mathbb{R}^n} f(\xi) dE(\xi).
\]

Let $A, B  \subseteq \mathbb{R}^n$ be   measurable subsets. The mapping $(E \otimes E)(A \times B)(x) = E(A) x E(B), x \in \cS_2$ defines an orthogonal projection on $\cS_2$. The mapping naturally extends to a spectral measure on the Borel sets of $\mathbb{R}^n \times \mathbb{R}^n$. We denote this measure by $F$.

Let $\phi: \mathbb{R}^n \times \mathbb{R}^n$ be a bounded Borel function. The mapping
\[
\cI_\phi = \int_{\mathbb{R}^n}\int_{\mathbb{R}^n} \phi(\xi, \tilde{\xi})dF(\xi, \tilde{\xi}),
\]
defines a bounded operator on  $\cS_2$. $\cI_\phi$ is called the {\it double operator integral} of $\phi$ with respect to the measure $E$. If $\cI_\phi: \cS_2 \cap \cS_p \rightarrow \cS_p$ admits a bounded extension to $\cS_p$, then we keep denoting this map with $\cI_\phi$. Suppose that $\phi(\xi, \tilde{\xi}) = f(\xi) - f(\tilde{\xi})$ for a Borel function $f$ on $\mathbb{R}^n$. Then,
\begin{equation}\label{EqnDoublOpCommutator}
\cI_\phi ( x ) =  f(\cA) x - x f(\cA), \qquad x \in \cS_2.
\end{equation}
Define $\delta(\xi, \tilde{\xi}) = 1$ if $\xi = \tilde{\xi}$ and $\delta(\xi, \tilde{\xi}) = 0$ if $\xi \not = \tilde{\xi}$. An element $x \in \cS_2$ is called {\it off-diagonal} (with respect to $E$) if
\[
\cI_{\delta}(x)  = \int_{\mathbb{R}^n}\int_{\mathbb{R}^n} \delta(\xi, \tilde{\xi}) F(\xi, \tilde{\xi})(x) = 0.
\]
For $x, y \in \cS_2$, we define a finite measure on  $\mathbb{R}^n \times \mathbb{R}^n$ by
\[
\nu_{y,x}(\Omega) := \tau(y\cI_{\chi_{\Omega}} x), \qquad \Omega \subseteq \mathbb{R}^n \times \mathbb{R}^n \textrm{ a Borel set}.
\]

\begin{rmk}
In case the operators $A_1, \ldots, A_n$ are unbounded, the analysis below becomes much more intricate. One has to treat the domains of the various operators and commutators very carefully, see for example \cite{PotSuk08}.
\end{rmk}

\begin{lem}\label{LemDuhamel}
Let $B_j$ $1 \leq j \leq n$ be a tuple of bounded commuting operators. Similarly, let $C_j$ $1 \leq j \leq n$ be a tuple of bounded commuting operators. Put $\cB = (B_1, \ldots, B_n)$ and $\cC = (C_1, \ldots, C_n)$. Then, for all $s \in \mathbb{R}^n$,
\[
\Vert e^{i s \cdot \cB} - e^{i s \cdot \cC} \Vert \leq \sum_{j=1}^n \vert s_j \vert \Vert B_j - C_j \Vert,
\]
where we use the notation $s\cdot \cB  = s_1 B_1  + \ldots + s_n B_n$.
\end{lem}
\begin{proof}
Using Duhamel's formula \cite{Win}, see also \cite[Lemma 8]{PotSuk12} with $r = 1$, one finds that for  any two self-adjoint operators $B$ and $C$ we have
\begin{equation}\label{EqnDuhamel}
\Vert e^{i B } - e^{i  C} \Vert \leq  \Vert B - C \Vert.
\end{equation}
Therefore, using first the triangle inequality and then (\ref{EqnDuhamel}),
\[
\begin{split}
\Vert e^{i s \cdot \cB} - e^{i s \cdot \cC} \Vert \leq & \sum_{j=1}^n  \Vert e^{i(s_1C_1 + \ldots + s_{j-1}C_{j-1} + s_j B_j + \ldots + s_n B_n)} -  e^{i(s_1C_1 + \ldots + s_{j}C_{j} + s_{j+1} B_{j+1} + \ldots + s_n B_n)} \Vert \\
 \leq &  \sum_{j=1}^n \vert s_j \vert \Vert B_j - C_j \Vert.
\end{split}
\]
\end{proof}

 For $l \in \mathbb{N}^\ast$, let
\[
U_l = \{ (\xi, \tilde{\xi}) \in \mathbb{R}^n \times \mathbb{R}^n \mid \Vert \xi - \tilde{\xi} \Vert_2 < \frac{1}{l} \}.
\]
 Then, $U_l$ is an open neighbourhood of the diagonal of $\mathbb{R}^n \times \mathbb{R}^n$.

\begin{prop}\label{PropDoubleOpIntError}
Let $p \in (1, \infty)$. Let $y \in \cS_p \cap \cS_2$ be such that there exists a $U_l, l \in \mathbb{N}^\ast$ such that we have $\int_{U_l} dF(\xi, \tilde{\xi})(y) = 0$. Let $\phi: \mathbb{R}^n \times \mathbb{R}^n \rightarrow \mathbb{R}$ be such that there exists a Schwartz function $\phi_0: \mathbb{R}^n \times \mathbb{R}^n \rightarrow \mathbb{R}$ for which $\phi(\xi, \tilde{\xi}) = \phi_0(\xi, \tilde{\xi})$  for every $(\xi, \tilde{\xi}) \in \mathbb{R}^n \times \mathbb{R}^n \backslash U_{l+1}$.  For $m \in \mathbb{Z}$ define a discrete spectral measure $E_m$
\[
E_m(\Omega) = \sum_{k \in \mathbb{Z}^n, {\rm s.t.} \frac{k}{m} \in \Omega} E\left(\left[ \frac{k_1}{m}, \frac{k_1+1}{m} \right) \times \ldots \times   \left[ \frac{k_n}{m}, \frac{k_n+1}{m} \right)\right), \qquad \Omega \subseteq \mathbb{R}^n \textrm{ a Borel set.}
\]
Consider the double operator integral $\cI_\phi$ of $\phi$ with respect to $E$. And similarly, let $\cI^m_\phi$ be the double operator integral of $\phi$ with respect to $E_m$. Then, for any $z \in \cS_{p'} \cap \cS_2$,
\[
\tau\left( z \cI^m_\phi y\right) \rightarrow \tau( z \cI_\phi y ), \qquad {\rm as} \quad m \rightarrow \infty.
\]
\end{prop}
\begin{proof}
Let $\cI_{\phi_0}$ and $\cI_{\phi_0}^m$ be the double operator integrals of $\phi_0$ with respect to $E$ and respectively $E_m$. Let $U_l^c$ be the complement of $U_l$ in $\mathbb{R}^n \times \mathbb{R}^n$. Our assumption on $y$ and $\phi_0$ implies that
\begin{equation}\label{EqnFirstZero}
\begin{split}
\cI_\phi(y) =&  \int_{\mathbb{R}^n} \int_{\mathbb{R}^n} \phi(\xi, \tilde{\xi}) dF(\xi, \tilde{\xi}) (y) =
\int_{U_l^c} \phi(\xi, \tilde{\xi}) dF(\xi, \tilde{\xi})(y) \\
= & \int_{U_l^c} \phi_0(\xi, \tilde{\xi}) dF(\xi, \tilde{\xi})(y)
= \int_{\mathbb{R}^n}  \int_{\mathbb{R}^n}  \phi_0(\xi, \tilde{\xi}) dF(\xi, \tilde{\xi})(y)
=   \cI_{\phi_0} (y).
\end{split}
\end{equation}
For $k \in \mathbb{Z}^n$ and $m \in \mathbb{N}$, we set $p_{k,m} = E_m(\frac{k}{m})$. Let $m$ be large (in fact $m  \geq \sqrt{n} l(l+1)$ suffices) and let $k,\tilde{k} \in \mathbb{Z}^n$ be such that $(\frac{k}{m}, \frac{\tilde{k}}{m}) \in U_{l+1}$. Then, since $\int_{U_l} dF(\xi, \tilde{\xi})(y) = 0$ we have that $p_{k,m} y p_{\tilde{k},m} = 0$. Hence, we compute
\begin{equation}\label{EqnSecondZero}
\begin{split}
\cI_\phi^m(y) =&  \sum_{k, \tilde{k} \in \mathbb{Z}^n}  \phi \left(\frac{k}{m}, \frac{\tilde{k}}{m}\right) p_{k,m} y p_{\tilde{k},m} =
 \sum_{k, \tilde{k} \in \mathbb{Z}^n, (\frac{k}{m}, \frac{\tilde{k}}{m}) \not \in U_{l+1}}  \phi\left(\frac{k}{m}, \frac{\tilde{k}}{m}\right) p_{k,m} y p_{\tilde{k},m} \\
= &  \sum_{k, \tilde{k} \in \mathbb{Z}^n, (\frac{k}{m}, \frac{\tilde{k}}{m}) \not \in U_{l+1}}  \phi_0\left(\frac{k}{m}, \frac{\tilde{k}}{m}\right) p_{k,m} y p_{\tilde{k},m}
=\sum_{k, \tilde{k} \in \mathbb{Z}^n}  \phi_0\left(\frac{k}{m}, \frac{\tilde{k}}{m}\right)p_{k,m} y p_{\tilde{k},m}
=   \cI_{\phi_0}^m (y).
\end{split}
\end{equation}

Let $\hat{\phi_0}$ be the Fourier transform of $\phi_0$. Then,
\[
\phi_0(\xi, \tilde{\xi}) = \int_{\mathbb{R}^n} \int_{\mathbb{R}^n} \hat{\phi}_0 (s, \tilde{s}) e^{is \cdot \xi} e^{i \tilde{s} \cdot \tilde{\xi}} ds d\tilde{s}.
\]
This implies that
\[
 \cI_{\phi_0}(y)  =  \int_{\mathbb{R}^n} \int_{\mathbb{R}^n}  \hat{\phi}_0(s, \tilde{s}) e^{i s\cdot \mathcal{A}  } y  e^{i \tilde{s}\cdot \mathcal{A}} ds d\tilde{s}.
\]
Let $A^m_j = \int_{\mathbb{R}^n} \xi_j dE_m(\xi)$ and set $\cA_m = (A^m_1, \ldots, A^m_n)$.
Then,
\begin{equation}\label{EqnOperatorIntegral}
\begin{split}
\cI_{\phi_0}(y) - \cI_{\phi_0}^m(y) = &\int_{\mathbb{R}^n} \int_{\mathbb{R}^n} \hat{\phi}_0(s, \tilde{s})( e^{i s \cdot \cA} y e^{i\tilde{s}\cdot \cA} - e^{i s \cdot \cA_m} y e^{i\tilde{s}\cdot \cA_m})ds d\tilde{s} \\ = &
\int_{\mathbb{R}^n} \int_{\mathbb{R}^n} \hat{\phi}_0(s, \tilde{s})( e^{is \cdot \cA} y (e^{i \tilde{s} \cdot \cA} - e^{i \tilde{s} \cdot \cA_m}) + (e^{i s \cdot \cA} - e^{is \cdot \cA_m}) y e^{i \tilde{s} \cdot \cA_m})ds d\tilde{s}
\end{split}
\end{equation}
We find the following estimates as $m \rightarrow \infty$, by respectively (\ref{EqnFirstZero}) and (\ref{EqnSecondZero}), then applying (\ref{EqnOperatorIntegral}) and Lemma \ref{LemDuhamel} and finally using that $\Vert A^m_j - A_j \Vert \leq \frac{1}{m}$,
\[
\begin{split}
\Vert \cI_{\phi}(y) - \cI_{\phi}^m(y)\Vert_2 = & \Vert \cI_{\phi_0}(y) - \cI_{\phi_0}^m(y) \Vert_2 \\
\leq & 2 \int_{\mathbb{R}^n} \int_{\mathbb{R}^n} \vert \hat{\phi}_0(s, \tilde{s})\vert  \sum_{j=1}^n \vert s_j \vert   \Vert A^m_j - A_j \Vert  \Vert y \Vert_2 ds d\tilde{s}\\
 \leq &
\frac{2}{m} \Vert y \Vert_2  \int_{\mathbb{R}^n} \int_{\mathbb{R}^n}  \vert \hat{\phi}_0(s, \tilde{s}) \vert\sum_{j=1}^n \vert s_j \vert   ds d\tilde{s}.
\end{split}
\]
Since $\phi_0$ is a Schwartz function, also the Fourier transform $\hat{\phi}_0$ is a Schwartz function. So the latter expression converges to 0.
\end{proof}

\begin{prop}[see Lemma 9 of \cite{PotSuk12}]\label{PropGausssian}
Let $\nu$ be a finite measure on $\mathbb{R}^n \times \mathbb{R}^n$. Let $\phi \in L^1(\mathbb{R}^n \times \mathbb{R}^n, \nu)$. Define,
\[
\phi_k(\xi, \tilde{\xi}) = \left( \sqrt{ \frac{k}{\pi}} \right)^n \int_{\mathbb{R}^n} e^{-k \eta \cdot \eta} \phi(\xi - \eta, \tilde{\xi} - \eta) d\eta.
\]
Then, $\Vert \phi_k - \phi \Vert_1 \rightarrow 0$ as $k \rightarrow \infty$.
\end{prop}
\begin{proof}
Since the bounded absolutely continuous functions are dense in $ L^1(\mathbb{R}^n \times \mathbb{R}^n, \nu)$ we may assume that $\phi$ is bounded and absolutely continuous. Let $\epsilon >0$ and choose $\delta > 0$ such that for every $(\xi, \tilde{\xi}),(\eta, \tilde{\eta})  \in \mathbb{R}^n \times \mathbb{R}^n$ with $\Vert (\xi, \tilde{\xi}) - (\eta, \tilde{\eta}) \Vert_2 < \delta$, we have  $\vert \phi(\xi, \tilde{\xi}) - \phi(\eta, \tilde{\eta}) \vert < \epsilon$. Then,
\[
\begin{split}
\Vert \phi_k - \phi \Vert_\infty = & \sup_{(\xi, \tilde{\xi}) } \left(\sqrt{ \frac{k}{\pi} } \right)^n \left|
\int_{\eta \in \mathbb{R}^n, \Vert \eta \Vert_2 \geq \delta} e^{-k \eta \cdot \eta} (\phi(\xi - \eta, \tilde{\xi} - \eta) - \phi(\xi, \tilde{\xi})) d\eta \right. \\ & +\left.
\int_{\eta \in \mathbb{R}^n, \Vert \eta \Vert_2 < \delta} e^{-k \eta \cdot \eta} (\phi(\xi - \eta, \tilde{\xi} - \eta) - \phi(\xi, \tilde{\xi})) d\eta
 \right| \\
\leq &    2\Vert \phi \Vert_\infty \left(\sqrt{ \frac{k}{\pi} } \right)^n
\int_{\eta \in \mathbb{R}^n, \Vert \eta \Vert_2 \geq \delta} e^{-k \eta \cdot \eta} d\eta   + \epsilon.
\end{split}
\]
The latter expression converges to 0 as $k \rightarrow \infty$. Since $\nu$ is finite, this implies that $\Vert \phi_k - \phi \Vert_1 \rightarrow 0$.
\end{proof}

\section{Commutator estimates}\label{SectCommutatorEstimate}

This section contains the main result of this paper. We prove that the best constant for operator Lipschitz inequalities and commutator estimates in Schatten-von Neumann classes are of order $\frac{p^2}{p-1}$.

In this section, $\xi, \tilde{\xi}$ are vectors in $\mathbb{R}^n$.
Suppose that $f: \mathbb{R}^n \rightarrow \mathbb{R}$ is a Lipschitz function. We define $\phi_f, \psi_f$ and $\phi_j, \psi_j, 1 \leq j \leq n$ on $\mathbb{R}^n \times \mathbb{R}^n$ by
\begin{equation}\label{EqnAuxFunctions}
\begin{split}
\psi_f(\xi, \tilde{\xi}) = f(\xi) - f(\tilde{\xi}),
&\qquad
\phi_f(\xi, \tilde{\xi}) =
\left\{
\begin{array}{ll}
 \frac{f(\xi) - f(\tilde{\xi})}{\Vert \xi - \tilde{\xi}\Vert_2} & \textrm{ if } \xi \not = \tilde{\xi}, \\
0 & \textrm{ if } \xi = \tilde{\xi},
\end{array}
\right. \\
 \psi_j(\xi, \tilde{\xi}) = \xi_j - \tilde{\xi}_j,
& \qquad
 \phi_j(\xi, \tilde{\xi}) = \left\{
\begin{array}{ll}
\frac{\xi_j - \tilde{\xi}_j}{\Vert \xi - \tilde{\xi}\Vert_2} & \textrm{ if } \xi \not = \tilde{\xi},\\
0 & \textrm{ if } \xi = \tilde{\xi}.
\end{array}
\right.
\end{split}
\end{equation}
Let $E$ be a spectral measure on $\mathbb{R}^n$ with compact support. Since the functions defined in (\ref{EqnAuxFunctions}) are all bounded on the support of $E$, the double operator integrals $\cI_{\phi_f}, \cI_{\phi_j}, \cI_{\psi_f}, \cI_{\psi_j}$ with respect to $E$ exist as bounded operators on $\cS_2$.

\begin{thm}\label{ThmIntegralEstimate}
Let $p \in (1, \infty)$. Let $f: \mathbb{R}^n \rightarrow \mathbb{R}$ be a Lipschitz function with $\Vert f \Vert_{{\rm Lip}} \leq 1$.  Let $y \in \cS_p \cap \cS_2$ be off-diagonal. For every $1 \leq j \leq n$, we have
\begin{equation}\label{EqnNormEstimate}
\Vert \cI_{\phi_f} \cI_{\phi_j}(y) \Vert_p \leq \frac{Cp^2}{p-1} \Vert y \Vert_p,
\end{equation}
for a constant $C$ that is independent of $p$, the spectral measure $E$ and the Lipschitz function $f$.
\end{thm}
\begin{proof}
 In order to prove the theorem, we first make three assumptions on $y$, $E$ and $f$. We show that each assumption can be made without loss of generality. Firstly, note that we assumed that $y$ is off-diagonal. The next assumption shows that we may in fact assume that   $y$ has no non-trivial part in a specific open neighbourhood of the diagonal.

\vspace{0.3cm}

{\bf Assumption 1.} For $l \in \mathbb{N}^\ast$, let $U_l = \{ (\xi, \tilde{\xi}) \in \mathbb{R}^n \times \mathbb{R}^n \mid \Vert \xi - \tilde{\xi} \Vert_2 < \frac{1}{l} \}$. It suffices to prove (\ref{EqnNormEstimate}) for $y \in \cS_p \cap \cS_2$ for which there exists a $l \in \mathbb{N}$  such that $\int_{U_l} dF(\xi, \tilde{\xi})(y) = 0$.

\vspace{0.3cm}

Let $y \in \cS_p \cap \cS_2$ be off-diagonal. Let
\[
\varphi_l(\xi, \tilde{\xi}) = \left\{
\begin{array}{ll}
1 & {\rm if } \: \Vert \xi - \tilde{\xi} \Vert_2 > \frac{1}{l}, \\
0 &  {\rm if } \: \Vert \xi - \tilde{\xi} \Vert_2 \leq \frac{1}{l},
\end{array}
\right.  \qquad
\varphi_\infty(\xi, \tilde{\xi}) = \left\{
\begin{array}{ll}
1 & {\rm if } \:   \xi \not = \tilde{\xi},\\
0 &  {\rm if } \: \xi = \tilde{\xi}.
\end{array}
\right.
\]
Then, $\int_{U_l} dF(\xi, \tilde{\xi})(\cI_{\varphi_l}(y)) = 0$. Furthermore, using respectively the definition of $\cI_{\varphi_l}$, the Lebesgue dominated convergence theorem and the fact that $y$ is off-diagonal, we find for every $z \in \cS_2$, as $l \rightarrow \infty$,
\begin{equation}\label{EqnLebesgue}
\begin{split}
 \tau ( z  \cI_{\phi_f} \cI_{\phi_j} \cI_{\varphi_l} y ) = & \int_{\mathbb{R}^n} \int_{\mathbb{R}^n} \phi_f \phi_j \varphi_l(\xi, \tilde{\xi}) d\nu_{z,y}(\xi, \tilde{\xi}) \\ \rightarrow & \int_{\mathbb{R}^n} \int_{\mathbb{R}^n}  \phi_f \phi_j \varphi_\infty(\xi, \tilde{\xi}) d\nu_{z,y}(\xi, \tilde{\xi}) \\ = & \tau(z\cI_{\phi_f} \cI_{\phi_j} y).
\end{split}
\end{equation}
Suppose that we have proved the  (\ref{EqnNormEstimate}) with $y$ replaced by $\cI_{\varphi_l}(y)$, in particular for a constant $C$ that is independent of $l$. Then, it follows from (\ref{EqnLebesgue}) that also (\ref{EqnNormEstimate}) holds for $y$. In all, this shows that we can make Assumption 1.

\vspace{0.3cm}

{\bf Assumption 2.} Suppose that $y \in \cS_p \cap \cS_2$ satisfies Assumption 1.  It suffices to prove Theorem \ref{ThmIntegralEstimate} under the condition that $E$ is a discrete spectral measure on $\mathbb{R}^n$ with support contained in $\frac{1}{m}\mathbb{Z}^n$ for some $m \in \mathbb{N}$.

\vspace{0.3cm}

We show that indeed Assumption 2 suffices to prove  Theorem \ref{ThmIntegralEstimate}.
 Let $f$ be an arbitrary Lipschitz function with $\Vert f \Vert_{{\rm Lip}} \leq 1$. Let $f_l, l \in \mathbb{N}$ be a sequence of Lipschitz functions with $\Vert f_l \Vert_{{\rm Lip}} \leq 1$, such that $f_l(\xi, \tilde{\xi}) = f(\xi, \tilde{\xi})$ for every $(\xi, \tilde{\xi}) \in [-l, l]^n$ and such that $f_l$ has compact support. Suppose that we have proved Theorem \ref{ThmIntegralEstimate} for all $f_l$.

Note that for every $\xi, \tilde{\xi} \in \mathbb{R}^n$ we have $\vert \phi_{f_l}(\xi, \tilde{\xi})  \phi_j(\xi, \tilde{\xi}) \vert \leq 1$ and $\phi_{f_l} \rightarrow \phi_f$ pointwise. The Lebesgue dominated convergence theorem hence entails that for every $z \in \cS_2$ we have, as $l \rightarrow \infty$,
\[
\begin{split}
\tau(z \cI_{\phi_{f_l}} \cI_{\phi_j} y) = & \int_{\mathbb{R}^n} \int_{\mathbb{R}^n} \phi_{f_l} (\xi, \tilde{\xi}) \phi_j(\xi, \tilde{\xi}) dv_{z,y}(\xi, \tilde{\xi}) \\
\rightarrow &  \int_{\mathbb{R}^n} \int_{\mathbb{R}^n} \phi_{f} (\xi, \tilde{\xi}) \phi_j(\xi, \tilde{\xi}) dv_{z,y}(\xi, \tilde{\xi}) \\
= & \tau(z\cI_{\phi_{f_l}} \cI_{\phi_j} y).
\end{split}
\]
From this limit, it follows that Theorem \ref{ThmIntegralEstimate} also holds for $f$. Hence, we may assume that $f$ has compact support.

 Let
\[
G_k(\xi) = \left(\sqrt{\frac{k}{\pi}}\right)^n  e^{-k (\xi \cdot \xi)},
\]
be a dilated Gaussian and put $f_k = G_k \ast f$. By Proposition \ref{PropGausssian},
\begin{equation}\label{EqnSchwartzApprox}
\begin{split}
\tau(z \cI_{\phi_{f_k}}   \cI_{\phi_j}  y) = & \int_{\mathbb{R}^n} \int_{\mathbb{R}^n} \phi_{f_k}(\xi, \tilde{\xi}) d\nu_{z,   \cI_{\phi_j}  y}(\xi, \tilde{\xi})  \\ \rightarrow  & \int_{\mathbb{R}^n} \int_{\mathbb{R}^n} \phi_{f}(\xi, \tilde{\xi}) d\nu_{z,   \cI_{\phi_j} y}(\xi, \tilde{\xi})   =  \tau(z \cI_{\phi_f}   \cI_{\phi_j}  y ), \qquad \textrm{ for every } \:  z \in \cS_{p'} \cap \cS_{2}.
\end{split}
\end{equation}
The function $f_k$ is Schwartz since $f$ has compact support. Furthermore, $\Vert f_k \Vert_{{\rm Lip}} \leq 1$. Suppose that Theorem \ref{ThmIntegralEstimate} is proved for all $f_k$, in particular with $C$ independent of $k$. Then, (\ref{EqnSchwartzApprox}) implies that Theorem \ref{ThmIntegralEstimate}  also holds for $f$. In all, this proves that we may assume that $f$ is a Schwartz function.

Now, assume that $f$ is a Schwartz function. Let $E$ be a spectal measure on $\mathbb{R}^n$. We define discretized spectral measures by setting
\[
E_m(\Omega) = \sum_{k \in \mathbb{Z}^n, {\rm s.t.} \frac{k}{m} \in \Omega} E\left(\left[ \frac{k_1}{m}, \frac{k_1+1}{m} \right) \times \ldots \times   \left[ \frac{k_n}{m}, \frac{k_n+1}{m} \right)\right), \qquad \Omega \subseteq \mathbb{R}^n \textrm{ a Borel set.}
\]
Let $\cI_{\phi_f}^m$ and $\cI_{\phi_j}^m$ be the  double operator integrals of $\phi_f$ and respectively $\phi_j$ with respect to the spectral measure $E_m$.   Let $U_l, l \in \mathbb{N}^\ast$ be the open neighbourhood of Assumption 1 for $y$. Since $f$ is Schwartz, there is a Schwartz function $\phi_{f,0}$ on $\mathbb{R}^n \times \mathbb{R}^n$ such that $\phi_{f,0}(\xi, \tilde{\xi}) = \phi_f(\xi, \tilde{\xi})$ for every $(\xi, \tilde{\xi}) \in \mathbb{R}^n \times \mathbb{R}^n \backslash U_{l+1}$.  It follows from Proposition \ref{PropDoubleOpIntError}    that
\begin{equation}\label{EqnLimitNot}
 \lim_{m \rightarrow \infty} \tau(z \cI^m_{\phi} \cI^m_{\phi_j} y) = \tau(z \cI_{\phi} \cI_{\phi_j} y), \qquad \textrm{ for every } z \in \cS_{p'} \cap \cS_{2}.
\end{equation}
Suppose that we have proved (\ref{EqnNormEstimate}) for $\cI^m_{\phi_f}, \cI^m_{\phi_j}$ and $y$ as in Assumption 1. In particular, the sequence in $m$ given by $\cI_{\phi_f}^m\cI_{\phi_j}^m y$ is bounded in $\cS_p$. Then, it follows from (\ref{EqnLimitNot}) that also (\ref{EqnNormEstimate}) holds for $\cI_{\phi_f}\cI_{\phi_j} y$. In all, this proves that without loss of generality we can make Assumption 2.

\vspace{0.3cm}

{\bf Assumption 3.} Let $y$ be as in Assumption 1 and let $E$ be a spectral measure as in Assumption 2. So the support of $E$ is contained in $\frac{1}{m} \mathbb{Z}^n$.  It suffices to prove Theorem \ref{ThmIntegralEstimate} under the condition that $f$ is a Lipschitz function with $\Vert f \Vert_{{\rm Lip}}\leq 1$ and such that there exists a $N \in \mathbb{N}$ such that
 $f$ maps $\frac{1}{m}\mathbb{Z}^n$  to $\frac{1}{mN}\mathbb{Z}$.

\vspace{0.3cm}

We prove that Assumption 3 is sufficient to conclude Theorem \ref{ThmIntegralEstimate}. Let $B_m: \mathbb{R}^n \rightarrow \mathbb{R}$ be a smooth function such that $B_m(0) = 1$ and the support of $B_m$ is contained in $[-\frac{1}{2m} , \frac{1}{2m} ]^n$.  Put
\[
f_{N}(\xi) =  f(\xi) + \sum_{k \in \mathbb{Z}^n}  \left( \frac{\lfloor N f(\frac{k}{m}) \rfloor}{N} - f\left(\frac{k}{m}\right) \right) B_m \left(\xi - \frac{k}{m}\right).
\]
Then, $f_{N}$ maps  $\frac{1}{m}\mathbb{Z}^n$  to $\frac{1}{mN}\mathbb{Z}$ and we have $\Vert f_N \Vert_{{\rm  Lip}} \leq 1 + \frac{1}{N} \Vert B_m \Vert_{\rm Lip}$ and $\Vert f - f_N \Vert_{{\rm Lip}} \leq \frac{1}{N} \Vert B_m \Vert_{{\rm Lip}}$. Moreover, for $z \in \cS_{p'} \cap \cS_{2}$,
\begin{equation}\label{EqnFunctionApprox}
\begin{split}
\tau(z (\cI_{\phi_{f_N}} - \cI_{\phi_{f}}) \cI_{\phi_j}y)
\leq &
\Vert (\phi_{f_N} - \phi_{f}) \phi_j  \Vert_\infty \Vert z \Vert_2 \Vert y \Vert_2
 \\ \leq &   \Vert f_N - f \Vert_{{\rm Lip}} \Vert \phi_j \Vert_\infty \Vert z \Vert_2 \Vert y \Vert_2,
\end{split}
\end{equation}
which converges to 0 as $N \rightarrow \infty$.
Suppose that (\ref{EqnOperatorIntegral}) is proved for all functions
$
g_N:= ( 1 + \frac{1}{N} \Vert B_m \Vert_{\rm Lip})^{-1} f_N,
$
(so that $\Vert g_N \Vert_{{\rm Lip}} \leq 1$).
 Then, in particular $\cI_{\phi_{f_N}} \cI_{\phi_j} y \in \cS_p$ is bounded in $N$. It follows from (\ref{EqnFunctionApprox}) that also $\cI_{\phi_{f}} \cI_{\phi_j}y$ is contained in  $\cS_p$ and satisfies the estimate (\ref{EqnNormEstimate}). In all, we conclude that without loss of generality, we can make Assumption 3.

\vspace{0.3cm}

We now prove Theorem \ref{ThmIntegralEstimate} under Assumptions 1, 2 and 3.
 For $k =(k_1, \ldots, k_n) \in \mathbb{Z}^n$, define the spectral projection
\[
p_k = E\left( \left[\frac{k_1}{m}, \frac{k_1+1}{m} \right) \times  \ldots \times \left[\frac{k_n}{m}, \frac{k_n+1}{m}\right) \right).
\]
For $\xi \in \mathbb{R}^n$ and $\mu \in \mathbb{R}$, define the unitary operator  acting $\cH$ by
\[
\begin{split}
u(\xi,  \mu) =&  \sum_{k \in \mathbb{Z}^n} e^{2\pi i mN( \frac{k}{m} \cdot \xi + f(\frac{k}{m})\mu)}p_k,
\end{split}
\]
For $y$ as in Assumption 1, put
\[
h_y =    u  \cdot y \cdot u^\ast.
\]
 Naturally, $h_y \in \cL^p_{\cS_p}(\mathbb{T}^{n+1})$.
 We consider $\mathbb{T}^{n+1}$ equipped with the normalized Lebesgue measure.  Then, $\Vert h_y \Vert_{\cL^p_{\cS_p}(\mathbb{T}^{n+1}) } = \Vert y \Vert_p$.

Fix $k, \tilde{k} \in  \mathbb{Z}^n$. Let $y \in p_k \cS_p p_{\tilde{k}}$ satisfy the condition of Assumption 1.
Since $\{ p_i\}_{i \in \mathbb{Z}^n}$ is a family of mutually orthogonal projections,
\begin{equation}\label{EqnHAX}
   h_{y}(\xi,  \mu) = e^{2\pi i mN( (\frac{k}{m} - \frac{\tilde{k}}{m})\cdot \xi  + (f(\frac{k}{m}) - f(\frac{\tilde{k}}{m}))\mu)}   y.
\end{equation}

Let $\delta_{ s, t}$ with $  s \in \mathbb{Z}^n, t \in \mathbb{Z}$ be the function on $\mathbb{Z}^{n+1}$ that attains the value 1 on $( s, t)$ and vanishes everywhere else. Taking the Fourier transform of $h_{y}$, we find
\[
\mathcal{F}_2\left( h_{y} \right) =  \delta_{N( k - \tilde{k}),   mN (f(\frac{k}{m}) -  f(\frac{\tilde{k}}{m}))} y \in  \cL^2_{\cS_p}(\mathbb{Z}^{ n+1}).
\]
 Using $\bar{m}_{j}$, the discretized version of   $m_{j}$, see  Theorem \ref{ThmDiscrete} and  Lemma \ref{LemMultiplierEven},
\begin{equation}\label{EqnHAXComputation}
\begin{split}
   T_{\bar{m}_{j}} h_{y} = & (\mathcal{F}_2^{-1} \circ \bar{m}_{j} \circ \mathcal{F}_2)  h_{y}\\
= & (\mathcal{F}_2^{-1} \circ \bar{m}_{j}) \left(
  \delta_{N( k - \tilde{k}) , mN (f(\frac{k}{m}) -  f(\frac{\tilde{k}}{m}))} y\right) \\
=  & \frac{f(\frac{k}{m}) - f(\frac{\tilde{k}}{m}) }{\Vert\frac{k}{m} - \frac{\tilde{k}}{m}\Vert_2} \frac{\frac{k_j}{m} - \frac{\tilde{k}_j}{m}}{ \Vert\frac{k}{m} - \frac{\tilde{k}}{m}\Vert_2}  \mathcal{F}_2^{-1} \left(
 \delta_{N( k - \tilde{k}),    mN (f(\frac{k}{m}) -  f(\frac{\tilde{k}}{m}))} y \right) \\
= &  \frac{f(\frac{k}{m}) - f(\frac{\tilde{k}}{m}) }{\Vert\frac{k}{m} - \frac{\tilde{k}}{m}\Vert_2} \frac{\frac{k_j}{m} - \frac{\tilde{k}_j}{m}}{ \Vert\frac{k}{m} - \frac{\tilde{k}}{m}\Vert_2}   h_{y}.
\end{split}
\end{equation}
On the other hand, recalling that $y \in p_k \cS_p p_{\tilde{k}}$,
\begin{equation}\label{EqnHAX2Computation}
\cI_\phi \cI_{\phi_j} y =  \frac{f(\frac{k}{m}) - f(\frac{\tilde{k}}{m}) }{\Vert\frac{k}{m} - \frac{\tilde{k}}{m}\Vert_2} \frac{\frac{k_j}{m} - \frac{\tilde{k}_j}{m}}{ \Vert\frac{k}{m} - \frac{\tilde{k}}{m}\Vert_2}  y.
\end{equation}
It follows from (\ref{EqnHAX}), (\ref{EqnHAXComputation}) and (\ref{EqnHAX2Computation}) that    for every    $y \in {\rm span}\left\{ p_k \cS_p p_{\tilde{k}}\mid k,\tilde{k} \in \frac{1}{m} \mathbb{Z}\right\}$ that satisfies Assumption 1,
\begin{equation}\label{EqnTMEquation}
T_{\bar{m}_{j}} h_{y} =   u \cdot  \cI_{\phi} \cI_{\phi_j} y \cdot  u^\ast.
\end{equation}
In particular,  $\cI_{\phi_f} \cI_{\phi_j} y\in \cS_p$ for all $1 \leq j \leq n$. Taking the norm in $\cL^p_{\cS_p}(\mathbb{T}^{n+1})$ on both sides of (\ref{EqnTMEquation}), one obtains the inequality
\[
\begin{split}
 \Vert \cI_{\phi_f} \cI_{\phi_j} y\Vert_p
= &
\Vert  u \cdot   \cI_{\phi_f} \cI_{\phi_j} y \cdot  u^\ast \Vert_{\cL^p_{\cS_p}(\mathbb{T}^{ n+1})}
\\ = & \Vert T_{\bar{m}_{j}} h_{y}  \Vert_{\cL^p_{\cS_p}(\mathbb{T}^{ n+1})} \\
\leq &
 \Vert T_{\bar{m}_{j}}: \cL^p_{\cS_p}(\mathbb{T}^{ n+1})   \rightarrow \cL^p_{\cS_p}(\mathbb{T}^{ n+1})  \Vert \Vert h_{y} \Vert_{\cL^p_{\cS_p}(\mathbb{T}^{ n+1})}
\\ = &
\Vert T_{\bar{m}_{j}}: \cL^p_{\cS_p}(\mathbb{T}^{ n+1}) \rightarrow \cL^p_{\cS_p}(\mathbb{T}^{ n+1})  \Vert \Vert y \Vert_p
\end{split}
\]
Using respectively Theorem \ref{ThmDiscrete},  Theorem \ref{ThmUMDEstimate} and Theorem \ref{ThmUMDConstantSp}, we continue the inequality:
\[
\begin{split}
 \Vert \cI_{\phi_f} \cI_{\phi_j} y \Vert_p  \leq&
\Vert T_{\bar{m}_{j}}: \cL^p_{\cS_p}(\mathbb{T}^{ n+1}) \rightarrow \cL^p_{\cS_p}(\mathbb{T}^{ n+1})  \Vert \Vert  y \Vert_p \\
\leq &
\Vert T_{m_{j}}: \cL^p_{\cS_p}(\mathbb{R}^{ n+1}) \rightarrow \cL^p_{\cS_p}(\mathbb{R}^{ n+1})  \Vert \Vert y \Vert_p \\
\leq &
C_1 \: {\rm UMD}_p(\cS_p)   \Vert  y \Vert_p \\
\leq & \frac{C_2 \: p^2}{p-1}   \Vert y \Vert_p ,
\end{split}
\]
where  $C_1$ and $C_2$ are constants  that are independent of the Lipschitz function $f$, the spectral measure $E$ and $p\in (1, \infty)$. Since ${\rm span}\left\{ p_k \cS_p p_{\tilde{k}}\mid k,\tilde{k} \in \frac{1}{m} \mathbb{Z}\right\}$ is dense in $\cS_p$, this concludes  the theorem.
\end{proof}

\begin{thm}\label{ThmOperatorLipschitz}
Let $f: \mathbb{R}^n \rightarrow \mathbb{R}$ be a Lipschitz function with $\Vert f \Vert_{{\rm Lip}} \leq 1$. Let $E$ be a spectral measure on $\mathbb{R}^n$ and let $\mathcal{A} = (A_1, \ldots, A_n)$ be defined as in (\ref{EqnArowDefinition}). Let $p \in (1,\infty)$. Let $x \in B(\cH)$ be such that for all $1 \leq j \leq n$ we have $[A_j, x] \in \cS_p$.  Then, also $[f(\cA), x] \in \cS_p$. Moreover, there exists a constant $C$ that is independent of $p\in(1, \infty)$, the spectral measure $E$ and the Lipschitz function $f$ such that
\begin{equation}\label{EqnFinalComEst}
\Vert [ f(\cA), x ] \Vert_p \leq \frac{Cp^2}{p-1} \sum_{j=1}^n \Vert [A_j, x] \Vert_p.
\end{equation}
\end{thm}
\begin{proof}
First assume that $x \in \cS_p \cap \cS_2$. In that case,
\begin{equation}\label{EqnCommutatorIntegral}
[f(\cA), x] = \cI_{\psi_f}(x) = \sum_{j=1}^n \cI_{\phi_f} \cI_{\phi_j} \cI_{\psi_j}(x) =  \sum_{j=1}^n \cI_{\phi_f} \cI_{\phi_j} [A_j, x].
\end{equation}
Here, the first and third equality are an application of (\ref{EqnDoublOpCommutator}). The second equality is a consequence of the fact that $g \mapsto \cI_g$ is an algebra homomorphism from the bounded Borel functions on $\mathbb{R}^n\times \mathbb{R}^n$ to the bounded operators acting on $\cS_2$. By Theorem \ref{ThmIntegralEstimate} we have for all $1 \leq j \leq n$ that  $ \cI_{\phi_f} \cI_{\phi_j} [A_j, x] \in \cS_p$ and moreover,
\begin{equation}\label{EqnIntermediateStatement}
   \Vert \cI_{\phi_f} \cI_{\phi_j} [A_j, x] \Vert_p \leq \frac{C\: p^2}{p-1}  \Vert  [A_j, x] \Vert_p.
\end{equation}
for a constant $C$ which is independent of $p$, $E$ and the function $f$ with $\Vert f \Vert_{{\rm Lip}} \leq 1$. Clearly, (\ref{EqnCommutatorIntegral}) and (\ref{EqnIntermediateStatement}) imply Theorem \ref{ThmOperatorLipschitz}.

Now, let $\mathfrak{X} = B(\cH)$ and $\mathfrak{N} = \cS_p$. For $x \in \mathfrak{X}$, let  $S_j(x) = [A_j, x],T(x) = [f(\cA), x]$. For $x \in \mathfrak{N}$, let  $R_j(x) = \cI_{\phi_f} \cI_{\phi_j} (x)$.  We will show that \cite[Lemma 5.1]{KPSS} is applicable.  The proof is exactly the same as   the final part of the proof of \cite[Theorem 5.3]{KPSS}.  We sketch it here. Firstly, for $x \in \cap_{j=1}^n \ker(S_j)$ we have for all $1 \leq j \leq n$ that $[A_j, x]= 0$. Hence, $[f(\cA), x] = 0$ and condition (1) of \cite[Lemma 5.1]{KPSS} is satisfied. Secondly, as explaind in the proof of \cite[Theorem 5.3]{KPSS} for any self-adjoint $Z \in B(\cH)$ the mapping $T_Z: \mathfrak{X} \rightarrow \mathfrak{X}: x \mapsto [Z,x]$ is hermitian and hence satisisfies $\ker(T_Z) \cap \overline{T_Z(\mathfrak{X})} = \{ 0 \}$. Hence,  condition (2) of \cite[Lemma 5.1]{KPSS} is satisfied. Finally,
\[
R_j \mathfrak{N} \subseteq \overline{\cI_{\phi_f} \cI_{\phi_j}  \mathfrak{N}} \subseteq \overline{ \cI_{\psi_f}  \mathfrak{N}} \subseteq  \overline{ T  \mathfrak{N}} \subseteq \overline{ T  \mathfrak{X}}.
\]
Here, the second inclusion follows from \cite[Lemma 2.4]{KPSS}. This proves that  condition (3) of \cite[Lemma 5.1]{KPSS} is satisfied. Applying \cite[Lemma 5.1]{KPSS} yields that
\begin{equation}\label{EqnExtendedEqn}
[f(\cA), x] =  \sum_{j=1}^n \cI_{\phi_f} \cI_{\phi_j} [A_j, x],
\end{equation}
for every $x \in B(\cH)$ such that for all $1 \leq j \leq n$ we have $[A_j, x] \in \cS_p$. Applying  Theorem \ref{ThmIntegralEstimate} to (\ref{EqnExtendedEqn})  yields Theorem \ref{ThmOperatorLipschitz}.
\end{proof}

\begin{rmk}
In \cite[Theorem 5.3]{KPSS}, Theorem \ref{ThmOperatorLipschitz} was proved with the weaker estimate
\[
C_p \leq \frac{C \: p^{16}}{(p-1)^8}.
\]
 In \cite{KPSS} the norms of the two double operator integrals $W_f$ and $V_j$ appearing the proof of  \cite[Theorem 5.3]{KPSS} are estimated separately with constants that do not give the same sharp result as in Theorem \ref{ThmOperatorLipschitz}. The novelty of our proof is the fact that we use the main result of \cite[Theorem 3.1]{GeiMonSak} (see Theorem \ref{ThmUMDEstimate}) to give a direct estimate of  $\cI_\phi \cI_{\phi_j}$.
\end{rmk}

\begin{rmk}
  The estimate $C_p \leq \frac{C \: p^2}{1-p}$ given in Theorem \ref{ThmOperatorLipschitz} is the best possible in the sense that in fact
\[
C_p \sim  \frac{C \: p^2}{1-p}.
\]
\end{rmk}

\begin{cor}
Let $f: \mathbb{R} \rightarrow \mathbb{R}$ be a Lipschitz function with $\Vert f \Vert_{{\rm Lip}} \leq 1$.  Let $p \in (1,\infty)$. Let $ X, Y \in B(\cH)$ be self-adjoint operators such that $X - Y \in \cS_p$. Then $f(X) - f(Y) \in \cS_p$. Moreover, there exists a constant $C$ that is independent of $p\in(1, \infty)$  and the Lipschitz function $f$ such that
\begin{equation}
\Vert f(X) - f(Y) \Vert_p \leq \frac{Cp^2}{p-1} \Vert X - Y \Vert_p.
\end{equation}
\end{cor}
\begin{proof}
Apply Theorem \ref{ThmOperatorLipschitz} to the case $n=1$ and with
\[
A_1 =  \left(
\begin{array}{ll}
X & 0 \\
0 & Y
\end{array}
\right),
\quad
x = \left(
\begin{array}{ll}
0 & 1 \\
1 & 0
\end{array}
\right).
\]
\end{proof}

Finally, we show that Theorem \ref{ThmOperatorLipschitz} implies
some variant of a weak $L^1$-type inequality. For $A \in B(\cH)$ a
compact operator and for $t \in [0,\infty)$, let
\[
\mu_t(A) = \inf \{ \Vert A p \Vert \mid p \in B(\cH) \textrm{ projection such that } \tau(p) \leq t \}.
\]
denote the decreasing rearrangement of singular values.
Let $\mathcal{L}^{1, \infty}$ be the weak $L^1$-space associated with $B(\cH)$.
It is defined as the space of all compact operators $A \in B(\cH)$ for which
\[
\Vert A \Vert_{1, \infty} = \sup_{t \in [0,\infty)} t  \mu_t(A) <
\infty.
\]
Consider also the space $M_{1, \infty}$ consisting of all compact operators $A$ such that
\[
\Vert A \Vert_{M_{1, \infty}} = \sup_{t \in [0,\infty)} \log(1+t)^{-1} \int_0^t \mu_s(A) ds  < \infty.
\]
We have a norm decreasing inclusion $\mathcal{L}^{1,\infty}\subseteq M_{1, \infty}$.   The following corollary is now a multi-variable version of the result obtained in \cite[Theorem 2.5 (ii)]{NazPel}.
\begin{cor}\label{CorWeakTypeEstimate}
Let $f: \mathbb{R}^n \rightarrow \mathbb{R}$ be a Lipschitz
function with $\Vert f \Vert_{{\rm Lip}} \leq 1$. Let $E$ be a
spectral measure on $\mathbb{R}^n$ and let $\mathcal{A} = (A_1,
\ldots, A_n)$ be defined as in (\ref{EqnArowDefinition}).  Let $x
\in B(\cH)$ be such that for all $1 \leq j \leq n$ we have $[A_j,
x] \in \cS_1$.  Then, also $[f(\cA), x] \in M_{1, \infty}$.
Moreover, there exists a constant $C$ that is independent of $E$
such that
\[
\Vert [ f(\cA), x ] \Vert_{M_{1, \infty}} \leq C \sum_{j=1}^n
\Vert [A_j, x] \Vert_1.
\]
\end{cor}
\begin{proof}
Put $T = [f(\mathcal{A}), x]$. Let $s >\!\!> 1$, set $p = \log(s)$ and $q = \frac{p}{p-1}$. Then, using the H\"older inequality, Theorem \ref{ThmOperatorLipschitz} and the inclusion $\cS_1 \subseteq \cS_q$, we find,
\[
\begin{split}
& \int_0^s \mu_t(T) ds \leq s^{\frac{1}{p}} \left( \int_0^s \mu_t(T)^q dt \right)^{\frac{1}{q}}  \\ \leq &
s^{\frac{1}{p}} \Vert T \Vert_q \leq s^{\frac{1}{p}} C q \sum_{j = 1}^n \Vert [ A_j, x] \Vert_q \leq
 s^{\frac{1}{p}} C q \sum_{j = 1}^n \Vert [ A_j, x] \Vert_1.
\end{split}
\]
Since $s^{\frac{1}{p}} q = e \frac{\log(s)}{\log(s)-1} \leq \log(s)$ for large $s$, we find that
 \[
\int_0^s \mu_t(T) ds \leq  C \log(s) \sum_{j = 1}^n \Vert [ A_j, x] \Vert_1,
\]
which implies that $\Vert T \Vert_{M_{1,\infty}} \leq C \sum_{j=1}^n \Vert [ A_j, x] \Vert_1$.
\end{proof}

\begin {rmk}
According to \cite[Theorem 4.5]{CRSS} (see also \cite[Theorem
2.1]{CGRS}) the $\Vert \cdot \Vert_{M_{1, \infty}}$-norm is
equivalent to the norm
\begin{equation}\label{EqnNormEquivalence}
\Vert A \Vert_{\zeta} = \limsup_{p\downarrow 1} (p-1) \Vert A \Vert_p.
\end{equation}
Using this observation, one may obtain an alternative proof of Corollary \ref{CorWeakTypeEstimate}.
\end{rmk}

\begin{rmk} The question whether the weak $L^1$-type
inequality holds, that is whether
\[
\Vert [ f(\cA), x ] \Vert_{1, \infty}  \leq
 C \sum_{j=1}^n \Vert [A_j, x] \Vert_1,
\]
 remains open.
\end{rmk}

\end{document}